\documentclass[11pt,a4paper,leqno,twoside, headinclude]{amsart}

\title[Isotropy formality of $(\mathbb{Z}_2\oplus \mathbb{Z}_2)$-symmetric spaces]{Equivariant formality of the isotropy action on $(\mathbb{Z}_2\oplus \mathbb{Z}_2)$-symmetric spaces}
\def\titl{Equivariant formality of the isotropy action on $(\mathbb{Z}_2\oplus \mathbb{Z}_2)$-symmetric spaces}
\def\auth{Manuel Amann and Andreas Kollross}
\date{February 18th, 2020}


\author{\auth}

\usepackage[english]{babel}

\usepackage[colorlinks,pdftex, plainpages=false]{hyperref}
\hypersetup{pdftitle=\titl, pdfauthor=\auth, pdftoolbar=false,
plainpages=false, hyperindex=true, pdfdisplaydoctitle=true}

\hypersetup{colorlinks,%
citecolor=black,%
filecolor=black,%
linkcolor=black,%
urlcolor=black,%
pdftex}

\usepackage{color}

\usepackage{amsmath, amssymb, amscd, amsthm}
\usepackage{stmaryrd}
\usepackage{latexsym}
\usepackage[all]{xy}
\usepackage{pb-diagram}
\usepackage{rotating}
\usepackage{multicol}
\usepackage{lscape}
\usepackage{wasysym}
\usepackage{longtable}
\usepackage{enumerate}
\usepackage{eepic}
\usepackage{url}

\xyoption{all}

\newtheorem{theo}{Theorem}[section]
\newtheorem{main}{Theorem}

\newtheorem*{main*}{Theorem}

\newtheorem*{mainprop*}{Proposition}

\newtheorem{mainconj}{Conjecture}

\newtheorem{prop}[theo]{Proposition}
\newtheorem{defi2}[theo]{Definition}
\newtheorem*{defi2*}{Definition}

\newenvironment{defi*}{\begin{defi2*}\normalfont}{\end{defi2*}}

\newenvironment{defin*}[1]{\begin{defi2*}[#1]\normalfont}{\end{defi2*}}
\newtheorem*{rem2*}{Remark}
\newenvironment{rem*}{\begin{rem2*}\normalfont}{\hfill$\boxbox$\end{rem2*}}
\newtheorem{rem2}[theo]{Remark}
\newenvironment{rem}{\begin{rem2}\normalfont}{\hfill$\boxbox$\end{rem2}}

\newtheorem{lemma}[theo]{Lemma}
\newtheorem{cor}[theo]{Corollary}
\newtheorem*{cor*}{Corollary}

\newtheorem{ques}[theo]{Question}
\newtheorem*{conj*}{Conjecture}
\newtheorem*{theo*}{Theorem}
\newtheorem*{ques*}{Question}
\newtheorem*{mi2}{Main Idea}

\newtheorem{ex2}[theo]{Example}

\newtheorem{exer2}[theo]{Exercise}

\newtheorem{alg2}[theo]{Algorithm}

\newcommand{\hh}{{\mathbb{H}}}                                     
\newcommand{\nn}{{\mathbb{N}}}                                     
\newcommand{\qq}{{\mathbb{Q}}}                                     
\newcommand{\rr}{{\mathbb{R}}}                                     
\newcommand{\pp}{{\mathbf{P}}}                                     
\newcommand{\zz}{{\mathbb{Z}}}                                     
\newcommand{\SO}{{\mathbf{SO}}}                                    
\newcommand{\U}{{\mathbf{U}}}                                      
\newcommand{\SU}{{\mathbf{SU}}}                                    
\newcommand{\Sp}{{\mathbf{Sp}}}                                    
\newcommand{\spl}{{\mathfrak{sp}}}                                 
\newcommand{\B}{{\mathbf{B}}}                                      
\newcommand{\C}{{\mathbf{C}}}                                      
\newcommand{\E}{{\mathbf{E}}}                                      
\newcommand{\F}{{\mathbf{F}}}                                      
\newcommand{\G}{{\mathbf{G}}}                                      
\newcommand{\Spin}{{\mathbf{Spin}}}                                
\newcommand{\dif} {{\operatorname{d}}}                             
\newcommand{\In} {{\,\subseteq\,}}                                 
\newcommand{\Ni} {\,{\supseteq}\,}                                 
\newcommand{\im} {{\operatorname{im\,}}}                           
\newcommand{\Hom}{{\operatorname{Hom}}}                            
\newcommand{\APL}{{\operatorname{A_{PL}}}}                         
\newcommand{\ADR}{{\operatorname{A_{DR}}}}                         
\newcommand{\rk}{{\operatorname{rk\,}}}                            
\newcommand{\diag}{{\operatorname{diag}}}                          
\newcommand{\co}{\colon\thinspace}                                 

\newcommand{\comment}[1]{}                                         
\newcommand{\xto}[1]{\xrightarrow{#1}}                             
\newcommand{\hto}[1]{\overset{#1}{\hookrightarrow}}                

\newcommand{\case}[1]{\textbf{Case #1.}}                           
\newcommand{\ack}{\noindent\textbf{Acknowledgements. }}            
\newcommand{\str}{\noindent\textbf{Structure of the article. }}    
\newcommand{\odd}{\textrm{odd}}                                    
\newcommand{\even}{\textrm{even}}                                  

\newenvironment{prf}{\begin{proof}[\textsc{Proof}]} {\end{proof}}     

\begin{document}

\maketitle \thispagestyle{empty}


\begin{abstract}
Compact symmetric spaces are probably one of the most prominent class of \emph{formal} spaces, i.e., of spaces where the rational homotopy type is a formal consequence of the rational cohomology algebra. As a generalisation, it is even known that their isotropy action is equivariantly formal.

In this article we show that $(\zz_2\oplus \zz_2)$-symmetric spaces are equivariantly formal and formal in the sense of Sullivan, in particular. Moreover, we give a  short alternative proof of equivariant formality in the case of symmetric spaces with our new approach.
\end{abstract}


\section*{Introduction}

An important notion in rational homotopy theory is the concept of \emph{formality}, which, roughly speaking, expresses the property of a space that its rational homotopy type can be formally derived from its rational cohomology algebra; that is, all rational information is contained already in the rational cohomology algebra. In particular, obstructions to formality like Massey products vanish. There are several prominent classes of formal manifolds: compact K\"ahler manifolds, symmetric spaces, and  homogeneous spaces of equal rank, just to mention a few. (Nonetheless, even amongst homogeneous spaces the lack of formality is generic---see \cite{Ama13} for several classes of non-formal examples.)

In this article we are interested in certain homogeneous spaces $G/K$ of compact connected Lie groups. An a priori unrelated concept is \emph{equivariant formality} of an action of a compact Lie group $K$ on a manifold $M$, which states that the \emph{Borel fibration}
\begin{align*}
M\hto{} M\times_K \E K \to \B K
\end{align*}
is \emph{totally non-homologous to zero}, i.e., the map induced by the fiber inclusion $H^*(M \times _K \E K;\qq)\to H^*(M;\qq)$ is a surjection.

In the special case of the \emph{isotropy action} of $K$ on $G/K$, i.e., the action given by left multiplication of the isotropy group $K$ on $G/K$, there is a well-known characterisation (see \cite[Theorem~1.4]{CF16}) which yields that equivariant formality of $K\curvearrowright G/K$ implies the formality of $G/K$. Accordingly, several classes of isotropy actions are known to be equivariantly formal, including the ones on symmetric spaces or, more generally, on $k$-symmetric spaces (see \cite{Goe12}, \cite{GN16}).

In this note we aim to find further classes of homogeneous spaces with equivariantly formal isotropy actions and to pave a way towards an answer to
\begin{ques}\label{ques01}
Let $G$ be a compact connected Lie group and let $\sigma$ be an (abelian) (Lie) group of automorphisms of~$G$. Does it hold true that the isotropy action on $G/G^\sigma_0$, where $G^\sigma_0$ denotes the identity component of the fixed point set of $\sigma$, is equivariantly formal?
\end{ques}
(Clearly, one may ask this question first for the formality of~$G/G_0^\sigma$.)

In this article we prove that a compact $(\zz_2\oplus \zz_2)$-symmetric space is equivariantly formal:
\begin{main}\label{mainthm}
Let $G$ be a compact connected Lie group and let $\sigma$ be a group of automorphisms of~$G$ isomorphic to~$\zz_2\oplus \zz_2$. Then the isotropy action on~$G/G^\sigma_0$ is equivariantly formal.
\end{main}

The proof proceeds by stepping through the classification of these spaces and first proving their formality. We then provide a sufficient criterion for the isotropy action on a formal homogeneous space to be equivariantly formal. Finally, it only remains to realise that the arguments we provided for formality are actually strong enough to already yield equivariant formality in the light of this new criterion.

We remark that our main result has recently been proven by Sam Hagh Shenas Noshari~\cite{Nos18}. His approach, however, is different from ours, directly proving equivariant formality of the isotropy action without drawing on the classification of the latter spaces. We hope, however, that our approach can be seen as a general approach which is useful in many similar situations. We illustrate this by quickly reproving the equivariant formality of the isotropy action on symmetric spaces.

As a consequence of Theorem~\ref{mainthm}, we have the following.
\begin{cor}
Question~\ref{ques01} can be answered in the affirmative whenever $\sigma$ is abelian and $|\sigma|\le7$.
\end{cor}
Indeed, the conjecture is trivial for the trivial group. For $|\sigma|=2$, it exactly states the well-known equivariant formality of the isotropy action on symmetric spaces~\cite{Goe12}. For $|\sigma|=4$ we have to distinguish the two cases $\sigma=\zz_4$ or $\sigma=\zz_2\oplus \zz_2$. The first case, as well as the cases $|\sigma|=3,5,6,7$ follow from the main results in~\cite{GN16}.

Note that due to \cite[p.~56, Proposition]{Lut81} a $(\zz_2\oplus \zz_2)$-symmetric space defined in analogy to a symmetric space (see \cite[Definition 1.1]{Lut81}) is already homogeneous. Observe that in \cite[Proposition 4.3]{Lut81} it is claimed that such a $(\zz_2\oplus \zz_2)$-structure is geometrically formal (the product of harmonic forms is again harmonic) once there exists a compatible metric. This would imply formality as well; in general, however, this does not occur (see for example \cite{KT03}).

\vspace{3mm}

\str  In Section~\ref{sec01} we review the classification of $(\zz_2\oplus \zz_2)$-symmetric spaces. In Section~\ref{sec02} we provide the necessary concepts from rational homotopy theory. The proof of the main result then starts with a case-by-case check of the formality of $(\zz_2\oplus \zz_2)$-symmetric spaces in Section~\ref{sec03}. In order to extend this to equivariant formality, we provide a criterion in Section~\ref{sec04}. We finish the proof of the equivariant formality of the isotropy action on $(\zz_2\oplus \zz_2)$-symmetric spaces in Section~\ref{sec05}. Moreover, we provide a new and short proof for the equivariant formality of the isotropy action on symmetric spaces (building on their well-known formality) in  Section~\ref{sec06}.

\vspace{3mm}

\ack
The authors thank Sam Hagh Shenas Noshari for several helpful discussions and Peter Quast for pointing out the article \cite{Lut81}. The authors also thank the referees for helping to improve the presentation of the article.

The first named author was supported both by a Heisenberg grant and his research grant AM 342/4-1 of the German Research Foundation; he is moreover associated to the DFG Priority Programme 2026.


\section{($\zz_2\oplus\zz_2)$-symmetric spaces}\label{sec01}

A compact homogeneous space~$M = G/H$ is called \emph{$(\zz_2\oplus\zz_2)$-symmetric} if there are two commuting automorphisms~$\alpha$, $\beta$ of~$G$ with $\alpha^2=\beta^2=1$ such that $(G^\alpha \cap G^\beta)_0 \In H \In G^\alpha \cap G^\beta$.

The classification of $(\zz_2\oplus \zz_2)$-symmetric spaces for simple~$G$ was carried out in~\cite[Theorem 14; Tables 1, 2, 3, 4]{BG08} for the classical cases and in~\cite[Theorems 1.2, 1.3; Table 1]{Kol09} for exceptional Lie groups~$G$.

We say that a $(\zz_2\oplus \zz_2)$-symmetric space $G/H$ is \emph{decomposable} if the Lie algebras of~$G$ and $H$ decompose as $\mathfrak{g} = \mathfrak{g}_1 \times \mathfrak{g}_2$ and $\mathfrak{h} = \mathfrak{h}_1 \times \mathfrak{h}_2$, where $\mathfrak{h}_i \In \mathfrak{g}_i$ and where $\mathfrak{g}_1,\mathfrak{g}_2$ are nontrivial---otherwise it is called \emph{indecomposable}. For our purposes, it suffices to consider indecomposable $(\zz_2\oplus \zz_2)$-symmetric spaces.
Note that the isotropy action of $G_1/H_1 \times G_2/H_2$ is equivariantly formal if and only if the isotropy actions of~$G_i/H_i$ are equivariantly formal for~$i=1,2$.

\begin{prop}\label{prop03}
Let $M=G/H'$ be an indecomposable $(\zz_2\oplus \zz_2)$-symmetric space, where $G$ is a connected compact Lie group and $H'$ is a connected closed subgroup such that its Lie algebra~$\mathfrak{h}'$ does not contain a simple ideal of~$\mathfrak{g}$. Then there is an automorphism~$\phi$ of~$G$ such that the pair~$(G,H)$, where $H=\phi(H')$, is one of the following:
\begin{enumerate}
\item \label{DecItemTorus} $G$ is a torus and $H$ a closed connected subgroup;
\item \label{DecItemSimple} $G$ is simple and $H = G$;
\item \label{DecItemSimpleZ} $G$ is simple and $G/H$ is a symmetric space in the usual sense;
\item \label{DecItemSimpleZZ} $G$ is simple and $G/H$ is a $(\zz_2\oplus\zz_2)$-symmetric space;
\item \label{DecItemTwo} up to finite coverings, $G = L \times L$, where $L$ is a simple compact Lie group and $H = \{ (g,g) \mid g \in L \}$;
\item \label{DecItemTwoZ} up to finite coverings, $G = L \times L$, where $L$ is a simple compact Lie group and $H = \{ (g,g) \mid g \in K \}$, where $K \In L$ is such that $L/K$ is a symmetric space in the usual sense;
\item \label{DecItemFour} up to finite coverings, $G = L \times L \times L \times L$ and $H = \{(g,g,g,g) \mid g \in L\}$, where $L$ is a simple compact Lie group.
\end{enumerate}
\end{prop}

\begin{proof}
For the Lie algebra of~$G$, we have the decomposition \[\mathfrak{g} = \mathfrak{g}_0 \oplus \mathfrak{g}_1 \oplus \dots \oplus \mathfrak{g}_k,\] where $\mathfrak{g}_0$ is an abelian ideal and $\mathfrak{g}_1, \dots, \mathfrak{g}_k$ are simple ideals.

The group $\sigma$ acts on~$\mathfrak{g}$ by automorphisms. Since the above decomposition is uniquely determined
and a Lie algebra automorphism maps the center of a Lie algebra onto itself, it follows that we have Case~(\ref{DecItemTorus}) if $\mathfrak{g}_0$ is nontrivial.

Otherwise, we may assume $\mathfrak{g}=\mathfrak{g}_1 \oplus \dots \oplus \mathfrak{g}_k$ is semisimple. Since any Lie algebra automorphism maps the semisimple part of a Lie algebra onto itself and any simple ideal~$\mathfrak{g}_i$ of~$\mathfrak{g}$ onto an isomorphic simple ideal~$\mathfrak{g}_j$ of~$\mathfrak{g}$, it follows that one can define an action of~$\sigma$ on the set $I := \{ 1, \dots, k \}$ by permutations by setting $\sigma(i) = j$ whenever $\sigma(X) \in \mathfrak{g}_j$ for $X \in \mathfrak{g}_i$.
The orbits of this action on~$I$ can have $1$, $2$, or $4$ elements. Thus, since we are dealing with indecomposable spaces, we have $k=1$, $2$, or $4$.

Assume $k = 1$. Then $\mathfrak{g}=\mathfrak{g}_1$ is simple and we have Cases~(\ref{DecItemSimpleZZ}), (\ref{DecItemSimpleZ}), or (\ref{DecItemSimple}), depending on whether the effectivity kernel of the action of  $\sigma$ on~$G$ has $1$, $2$, or $4$ elements.

Now assume $k = 2$. Then $\mathfrak{g}=\mathfrak{g}_1 \oplus \mathfrak{g}_2$, where $\mathfrak{g}_1,\mathfrak{g}_2$ are two isomorphic simple ideals.
Let $\alpha, \beta$ denote two generators of~$\sigma \cong \zz_2\oplus \zz_2$. Then at least one of $\alpha, \beta$ maps $\mathfrak{g}_1$ to $\mathfrak{g}_2$, say $\alpha(\mathfrak{g}_1)=\mathfrak{g}_2$.
Now the other generator $\beta$ either leaves $\mathfrak{g}_1$ invariant or maps~$\mathfrak{g}_1$ to~$\mathfrak{g}_2$.
In the first case define $\gamma := \beta$, in the latter case define $\gamma := \alpha \circ \beta$. In either case, $\alpha$ and $\gamma$ are two generators of~$\sigma$ and $\gamma|_{\mathfrak{g}_1}$ is an automorphism of~$\mathfrak{g}_1$.
Since $\sigma$ is an abelian group we have $\gamma = \alpha \circ \gamma \circ \alpha^{-1}$.
In particular, we have \[\gamma|_{\mathfrak{g}_2}  = \alpha \circ \gamma|_{\mathfrak{g}_1} \circ \alpha^{-1}|_{\mathfrak{g}_2}.\]
If now $\gamma$ acts on~$\mathfrak{g}_1$ as the identity, it also acts on~$\mathfrak{g}_2$ as the identity and we are in Case~(\ref{DecItemTwo}). Otherwise, $\gamma$ acts as a nontrivial involutive automorphism on~$\mathfrak{g}_1$; using now $\alpha$ to identify $\mathfrak{g}_1$ and $\mathfrak{g}_2$, we have $\gamma|_{\mathfrak{g}_1} = \gamma|_{\mathfrak{g}_2}$ and we are in Case~(\ref{DecItemTwoZ}).

If we have $k=4$, we are in the remaining Case~(\ref{DecItemFour}).
\end{proof}

\section{Tools from rational homotopy theory}\label{sec02}

Rational homotopy theory deals with the so-called rational homotopy type of nilpotent spaces. In this short note we shall not try to review this theory, but we point the reader to \cite{FHT01} and \cite{FOT08} as references. The article builds on Sullivan models $(\Lambda V,\dif)$ of nilpotent spaces. Recall the main theorem of rational homotopy theory, which, up to duality, identifies the underlying graded vector space $V$ with the rational homotopy groups $\pi_{\geq 2}(X)\otimes \qq$ of $X$ (with a special interpretation using the Malcev completion in degree $1$).

In particular, the construction of a Sullivan model of homogeneous spaces upon which we draw heavily can be found in \cite[Chapter 15 (f)]{FHT01} or \cite[Chapter 3.4.2]{FOT08}. Recall that finite $H$-spaces have minimal Sullivan models which are graded free exterior algebras generated by finitely many elements. Due to the long exact sequence in homotopy applied to the universal fibration $\B G \to \E G\to G$ of a compact Lie group $G$ (together with the fact that $V$ is dual to rational homotopy groups and a lacunary argument for the differential) we know that $H^*(\B G;\qq)$ is actually a free polynomial algebra generated in even degrees. A usually non-minimal Sullivan model for the homogeneous spaces $G/K$ of compact connected Lie groups with $K$ connected (these spaces are known to be simple, see \cite[Proposition 1.62, p.~31]{FOT08}, whence accessible to rational homotopy theory, and they are the \emph{only} homogeneous spaces which we consider throughout this article) is due to Cartan and given by
\begin{align*}
(H^*(\B K)\otimes \Lambda V_G,\dif)
\end{align*}
where $(\Lambda V_G,0)\cong H^*(G;\qq)$ and the differential $\dif$ vanishes on $H^*(\B K)$. Thus the differential is induced (and extended as a derivation) by its behaviour on a homogeneous basis $(v_i)$ of $V_G$. One finds
\begin{align*}
\dif v=H^*(\B \phi)(v^{(+1)}),
\end{align*}
where $\phi\co K\hto{} G$ is the inclusion and $v^{(+1)}$ denotes the suspension, a $(+1)$-degree shift of $v$. Indeed, for this, from the long exact homotopy sequence of the universal bundle $G\hto{} \E G\to \B G$ one easily recalls from the aforesaid together with the main theorem of rational homotopy theory relating rational homotopy groups to generators of a minimal Sullivan model that a minimal Sullivan model of $\B G$ is of the form $(\Lambda V^{+1}_G,0)$. That is, it is a polynomial algebra identical to its cohomology generated by the suspensions of the $v$, the generators of $V_G$.

Throughout this article we will use the above models of the homogeneous spaces and always use rational coefficients in cohomology.
Recall further the concept of a \emph{pure space} (see \cite[p.~435]{FHT01}); that is, a space admitting a \emph{pure Sullivan model}, i.e., a model with the following decomposition of its minimal Sullivan model: let $(\Lambda V,\dif)$ be such a model, then $\dif|_{V^\even}=0$ and $\dif V^\odd\In \Lambda V^\even$. (We may extend this definition to non-minimal Sullivan models.) The model of a homogeneous space constructed above clearly is of this form (and hence so is its minimal model). Thus we recall
\begin{lemma}[{\cite[Proposition 32.10, p.~444]{FHT01}}]
Let $(\Lambda V,\dif)$ be a pure Sullivan algebra with finite-dimensional cohomology satisfying that $\dim V^\even=\dim V^\odd$; then a homogeneous basis of $V^\odd$ is mapped to a regular sequence in $\Lambda V^\even$, and the algebra is formal, in particular.
\end{lemma}
\vspace{5mm}

Let us finally recall the concept of formality here. A nilpotent space is formal if there is a chain of quasi-isomorphisms, i.e., of morphisms of cochain algebras inducing isomorphisms on cohomology, between the cochain algebra of polynomial differential forms $\APL(X)$ and the cohomology $H^*(X;\qq)$. As formality does not depend on the field extension of $\qq$ (see \cite[Corollary 6.9, p.~237]{HS79}) we may replace $\APL(X)$ by de Rham forms $\ADR(M)$ on a smooth manifold. Another simplification may be formulated using minimal Sullivan models $(\Lambda V,\dif)\xto{\sim} \ADR(M)$ (now over the reals). Then $M$ is formal if and only if there is a quasi-isomorphism $(\Lambda V,\dif)\xto{\sim} H^*(M;\rr)$.

Note that a finite product is formal if and only if so are its factors---we shall actually only use the completely obvious implication that the formality of the factors implies the formality of the product.

It is well-known that a homogeneous space $G/H$ (with $G$, $H$ compact and $H$ connected) is formal if and only if $G/T_H$ is formal, where $T_H$ is a maximal torus of~$H$. This can be easily derived from \cite[Theorem A]{AK12} applied to the fibration
\begin{align*}
H/T_H\hto{} G/T_H\to G/H
\end{align*}
using the fact that $H/T_H$ satisfies the Halperin conjecture, i.e., its cohomology (basically for degree reasons) has no non-trivial derivations of negative degree --- and was known before (see \cite[Remark, p.~212]{Oni94}).

\begin{cor}\label{cor03}
Let $H,H'\In G$. Suppose that the maximal torus $T_H\In H$ is conjugate in $G$ to the maximal torus $T_{H'}\In H'$. Then $G/H$ is formal if and only if $G/H'$ is formal.
\end{cor}
\begin{prf}
Conjugation induces a diffeomorphism $G/T_H\cong G/T_{H'}$.
\end{prf}

Recall the notion of an \emph{elliptic space}, which is a space that is nilpotent with finite dimensional rational homotopy and finite dimensional rational cohomology. For pure elliptic models formality implies a well-known strong structural result (for example, see \cite{Ama13}, \cite[Proposition 32.10, p.~444]{FHT01}):

\begin{prop}\label{prop02}
Let $(\Lambda V,\dif)$ be a pure elliptic minimal Sullivan algebra. Then up to isomorphism it is of the form
\begin{align*}
(\Lambda V,\dif)\cong (\Lambda V',\dif) \otimes (\Lambda\langle x_i\rangle_{1\leq i\leq k}, 0)
\end{align*}
with $(\Lambda V',\dif)$ elliptic, with $k\geq 0$, and with $\deg x_i$ necessarily odd.
In particular, it is formal if and only if $(\Lambda V',\dif)$ has positive Euler characteristic. In this case $\dif$ maps a homogeneous basis of $V'$ to a regular sequence in $\Lambda (V')^\even$.

This is necessarily the case if $\dim V^\even=\dim V^\odd$ (whence $k=0$).

In the case of a compact homogeneous space~$G/K$, the number~$k$ is the difference~$\rk G- \rk K$.
\end{prop}
Clearly, this is equivalent to a respective decomposition of the cohomology of the pure model into a factor of positive Euler characteristic and a free exterior algebra generated in odd degrees.
In this article, we will call the number~$\rk G- \rk K$ the \emph{corank} of~$K \In G$.
We call a rationally elliptic space of positive Euler characteristic \emph{positively elliptic}.

\section{Formality: Stepping through the cases}\label{sec03}

\subsection{Reducing to the case of simple $G$}\label{secfrt}

\begin{lemma}
The spaces $G/H$, where $(G,H)$ is a pair as given in Proposition~\ref{prop03}, are formal if $G$~is non-simple.
\end{lemma}
\begin{prf}
\textbf{(\ref{DecItemTorus})} If $G$ is a torus and $H$ a closed connected subgroup, then it is again a torus, and so will be $G/H$. Hence $G/H$ is formal.

\textbf{(\ref{DecItemTwo})} If $G = L \times L$, where $L$ is a simple compact Lie group and
$H = \{ (g,g) \mid g \in L \}$, then the quotient is just diffeomorphic to the antidiagonal, which again is diffeomorphic to $L$, hence formal

\textbf{(\ref{DecItemTwoZ})} If $G = L \times L$, where $L$ is a simple compact Lie group and
$H = \{ (g,g) \mid g \in K \}$ with $K \In L$ such that $L/K$ is a symmetric space, then we can argue as follows: The inclusion of~$H$ in~$G$ is the composition of the inclusions $H=K\In L \In L\times L=G$, where the latter is a diagonal inclusion. It follows that the Sullivan model of $G/H$ is of the form
\begin{align*}
(\Lambda \langle x_1,\dots , x_l\rangle,0) \otimes (A,\dif),
\end{align*}
where the first factor is a model of the antidiagonal $\{(l,l^{-1}) \mid l\in L\}$ in $L\times L$, and $(A,\dif)$ is a model of $L/K$. Hence it is a tensor product of formal algebras and $G/H$ is formal.

\textbf{(\ref{DecItemFour})}  If $G = L \times L \times L \times L$ and $H = \{(g,g,g,g) \mid g \in L\}$, where $L$ is a simple compact Lie group, then $G/H$ is diffeomorphic to $L\times L\times L$, realised as a complement of the diagonal in $L\times L \times L \times L$, hence formal.
\end{prf}

\subsection{Coverings}

We remark that if $G/H$ is a homogeneous space of compact Lie groups $G$, $H$, and both $G$ and $H$ are connected, then $G/H$ is a simple space (see \cite[Proposition 1.62, p.~31]{FOT08}), and the techniques of rational homotopy theory apply. In particular, a Sullivan model of $G/H$ can be constructed as described in Section~\ref{sec02}. However, this implies that if $G'/H'$ is a homogeneous space of connected Lie groups and a finite covering of $G/H$ (such that both $G$, $H$ are connected and $G'$ covers~$G$ in such a way that $H'$ covers $H$), then the model we construct for it is identical to the one of $G/H$. Indeed, for this we note that the inclusions of the maximal tori $T_H\hto{} T_G$ and $T_{H'}\hto{} T_{G'}$ which completely determine the models (together with the given Weyl groups, which remain the same under finite coverings) are uniquely determined by local data, namely, by the inclusions of the corresponding Lie algebras $\mathfrak{t}_H=\mathfrak{t}_{H'}\hto{} \mathfrak{t}_G=\mathfrak{t}_{G'}$ due to the commutativity and surjectivity properties of the exponential map on compact Lie groups.

Hence, the spaces $G/H$ and $G'/H'$, which may differ by a finite covering, are rationally not distinct as long as all groups are connected; hence, in the following, we may pick a representative out of the each equivalence class.

$(\zz_2\oplus \zz_2)$-symmetric spaces $G/H$ where $G$ is a simple compact Lie group have been classified in~\cite{BG08} and~\cite{Kol09}.

\subsection{Homogeneous spaces of equal rank groups} These spaces are positively elliptic spaces with cohomology concentrated in even degrees given as the quotient of a polynomial algebra by a regular sequence (see \cite[Proposition 32.2, p.~436]{FHT01}). In particular, they are formal.

The following $(\zz_2\oplus \zz_2)$-spaces fall in this category:
\begin{itemize}
\item $\SU(a+b+c)/{\mathbf S}(\U(a)\U(b)\U(c))$;
\item $\SU(a+b+c+d)/{\mathbf S}(\U(a)\U(b)\U(c)\U(d))$;
\item $\SO(a+b+c)/\SO(a)\SO(b)\SO(c)$, where at most one of~$a,b,c$ is odd;
\item $\SO(a+b+c+d)/\SO(a)\SO(b)\SO(c)\SO(d)$, where at most one of~$a,b,c,d$ is odd;
\item $\Sp(a+b+c)/\Sp(a)\Sp(b)\Sp(c)$;
\item $\SO(2a+2b)/\U(a)\U(b)$;
\item $\Sp(a+b+c+d)/\Sp(a)\Sp(b)\Sp(c)\Sp(d)$;
\item $\E_6/\U(3)\U(3)$;
\item $\E_6/\Spin(6)\Spin(4)\U(1)$;
\item $\E_6/\U(5)\U(1)$;
\item $\E_6/\Spin(8)\SO(2)\U(1)$;
\item all $(\zz_2\oplus \zz_2)$-symmetric spaces~$G/H$ where $G = \E_7$, $\E_8$, $\F_4$, or $\G_2$ except for $\E_7/\F_4$, $\E_7/\Sp(4)$, $\E_7/\Spin(8)$.
\end{itemize}
For concrete inclusions here and in the following subsections see the cited classifications \cite[Theorem~14; Tables~1, 2, 3, 4]{BG08} and \cite[Theorems~1.2, 1.3; Table~1]{Kol09}.

\subsection{Spaces sharing the maximal torus with a symmetric space}

As a consequence of Corollary~\ref{cor03} we note that any space~$G/H$ such that $H$ shares its maximal torus with some~$H'$ such that $G/H'$ is a symmetric space is formal.

We apply this to the following $(\zz_2\oplus \zz_2)$-symmetric spaces:
\begin{itemize}
\item $\SU(2k)/\U(k)$, where~$\U(k) \In \Sp(k) \In \SU(2k)$;
\item $\SU(n+m)/\SO(n)\SO(m)$, where at most one of $n$, $m$ is odd;
\item $\SU(2a+2b)/\Sp(a)\Sp(b)$, where $\Sp(a)\Sp(b) \In {\mathbf S(\U(2a)\U(2b))}$;
\item $\SO(a+b+c)/\SO(a)\SO(b)\SO(c)$, where~$a,b$ are odd and $c$ is even, we have $\SO(a)\SO(b)\SO(c) \In \SO(a)\SO(b+c) \In \SO(a+b+c)$;
\item $\SO(a+b+c+d)/\SO(a)\SO(b)\SO(c)\SO(d)$, where~$a,b$ are odd and $c,d$ are even, we have $\SO(a)\SO(b)\SO(c)\SO(d) \In \SO(a)\SO(b+c+d) \In \SO(a+b+c+d)$;
\item $\Sp(a+b)/\U(a)\U(b)$, here we have $\U(a)\U(b) \In \U(a+b) \In \Sp(a+b)$;
\item $\E_6/\U(4)$, where $\U(4) \In \pp\Sp(4) \In \E_6$;
\item $\E_6/\Spin(9)$;
\item $\E_6/\Sp(2)\Sp(2)$, where $\Sp(2)\Sp(2) \In \pp\Sp(4) \In \E_6$;
\item $\E_6/\Sp(3)\Sp(1)$, where $\Sp(3)\Sp(1) \In \pp\Sp(4) \In \E_6$.
\end{itemize}
We observe that $\Spin(9)$ and $\Sp(2)\Sp(2)$ share a common maximal torus (up to covering); so do $\SO(n)\SO(m)$ and $\SO(n+m)$.
We see that there are a symmetric space $\E_6/\pp\Sp(4)$ and inclusions
\begin{align*}
(\U(4)\In )\pp\Sp(4) \Ni \Sp(2)\Sp(2).
\end{align*}
All of these groups have rank $4$, whence they share a common maximal torus in $\E_6$.

As for sharing the maximal torus with $\Spin(9)$ we draw on the symmetric space $\E_6/\Spin(10)\U(1)$. The group $\Sp(2)\Sp(2)$ in the form of $\Spin(5)\Spin(5)$ is contained in this $\Spin(10)$-subgroup and shares a maximal torus with $\Spin(9)$ in $\E_6$. (This follows easily from the block inclusions of $\SO(5)\SO(5)\In \SO(10)\supseteq \SO(9)$.)

\subsection{The remaining spaces}\label{subsec01}
The remaining $(\zz_2\oplus \zz_2)$-spaces not covered by the above two cases are the following:

\begin{itemize}
\item $\SU(2n)/\SU(n)$;
\item $\SU(n+m)/\SO(n)\SO(m)$, where $n$ and $m$ are odd;
\item $\SO(a+b+c)/\SO(a)\SO(b)\SO(c)$, where 
all three of~$a,b,c$ are odd;
\item $\SO(a+b+c+d)/\SO(a)\SO(b)\SO(c)\SO(d)$, where three or four of~$a,b,c,d$ are odd;
\item $\SO(2n)/\SO(n)$;
\item $\SO(4n)/\Sp(n)$;
\item $\Sp(2n)/\Sp(n)$;
\item $\Sp(n)/\SO(n)$;
\item $\E_7/\Sp(4)$;
\item $\E_7/\SO(8)$;
\item $\E_7/\F_4$.
\end{itemize}

\subsection{Diagonal double block inclusions}\label{subsec02}
In the following we shall deal with the remaining cases.

\case{1} Let us consider the space $\Sp(2n)/\Sp(n)$, which is given by the two-block diagonal inclusion of $\Sp(n)$ into $\Sp(2n)$, see~\cite[Section~4.2.2]{BG08}.
We use the convention
\[
\spl(n) =\left \{
\left( \left.
  \begin{array}{rc}
    A & B \\
    -\bar B & \bar A \\
  \end{array}
\right) \right|
A,B \in \C^{n \times n}, A=-\bar{A}^t, B=B^t
\right \}.
\]
Now using the inclusions $\spl(n) \In \spl(n) \oplus \spl(n) \In \spl(2n)$, we obtain
the subalgebra
\[
\left \{
\left( \left.
         \begin{array}{rrcc}
           A & 0 & B & 0 \\
           0 & A & 0 & B \\
           -\bar B & 0 & \bar A & 0 \\
           0 & -\bar B & 0 & \bar A \\
         \end{array}
       \right)
 \right|
A,B \in \C^{n \times n}, A=-\bar{A}^t, B=B^t
\right \}  \In \spl(2n).
\]

The standard inclusion $\Sp(n)\hto{} \Sp(2n)$ is rationally $2n$-connected. As in the following cases for corresponding inclusions of unitary and orthogonal groups, this can be verified by direct computation or by using one of \cite[p.~220]{FHT01} and \cite[Chapter~3, proofs of 3.1, 3.15, 3.17]{MT91}. For the convenience of the reader we sketch the computation. We prove the analogous statement for the induced map between classifying spaces. For this we first identify rational homotopy groups with cohomology generators and second we describe the latter as polynomials in the cohomology generators $t_i^2$ of the classifying space of the maximal torus invariant under the action of the Weyl group. {That is, these generators/Pontryagin classes identify with elementary symmetric polynomials in the $t_i^2$ under the left and right hand maps
\begin{align*}
H^*(\B \Sp(n))\hto{} H^*(\B T_{\Sp(n)}) \xleftarrow{} H^*(\B T_{\Sp(2n)}) \hookleftarrow{} H^*(\B \Sp(2n))
\end{align*}
with the middle map being the obvious projection induced by $t_{n+i}\mapsto t_i$ for $0<i\leq n$. It follows that corresponding generators are mapped to each other. This may be seen  more easily using power sums $\sum_{i=1}^k (t_i^2)^k$ as generators instead of elementary symmetric polynomials.

Let us now provide the formality argument. The connectedness property implies that the induced map on rational homotopy groups of the considered inclusion $\Sp(n)\hto{} \Sp(2n)$ is multiplication by two in degrees $1\leq i \leq 4n-1$ and, clearly, the projection to zero in degrees $4n+1 \leq i\leq 8n-1$, as the rational homotopy groups of $\Sp(n)$ are concentrated below degree $4n+1$. This implies that $\pi_*(\Sp(2n)/\Sp(n))\otimes \qq$ is concentrated in odd degrees. More precisely, considering the model
\begin{align*}
(\Lambda (V_{\B \Sp(n)} \oplus V_{\Sp(2n)}),\dif)
\end{align*}
of this homogeneous space where 
$(\Lambda V_{\B \Sp(n)},0)$ and $(\Lambda V_{\Sp(2n)},0)$ are minimal models for $\B \Sp(n)$ and $\Sp(2n)$ respectively. We deduce
\begin{align*}
(\Lambda (V_{\B \Sp(n)} \oplus V_{\Sp(2n)}),\dif)\simeq (\Lambda \langle v_{4n+1}, \dots, v_{8n-1}\rangle,0),
\end{align*}
with generators $v_i$ of degree $i$ corresponding to the rational homotopy groups of $\Sp(2n)$ in degree $i$.

This can be seen as follows:
The linear part $\dif_0$ of the differential $\dif$ corresponds to the transgression in the dual long exact homotopy sequence of the fibration $\Sp(2n) \hto{} \Sp(2n)/\Sp(n) \to \B\Sp(n)$. The generators of the underlying vector spaces of the minimal models identify with the rational homotopy groups up to duality. It follows that the vector spaces $V_{\B \Sp(n)}$ and $V_{\Sp(2n)}$ are graded isomorphic via $\dif_0$ (raising degree by $+1$) in degrees below $4n+1$. In other words, up to a contractible algebra, a Sullivan model of the homogeneous space is concentrated in degrees at least $4n+1$, i.e., so is its minimal Sullivan model. Moreover, without loss of generality, the differential on the remaining generators becomes $0$. Consequently, $\Sp(2n)/\Sp(n)$ is formal.

\vspace{5mm}

\case{2}  The space $\SU(2n)/\SU(n)$. Here, the embedding of $\SU(n)$ into $\SU(2n)$ is given by $K=\{\diag(X,X) \mid X \in \SU(n) \}$; see the last case in~\cite[Section~4.3]{BG08}. From the same sources or with the analogous arguments as in Case 1 we derive that the standard inclusion $\SU(n)\hto{} \SU(2n)$ is (rationally) $(2n+1)$-connected. Again, it follows that our considered inclusion induces multiplication by $2$ on rational homotopy groups in degrees at most $2n+1$ and the trivial map in higher degrees. As for the corresponding model of the homogeneous space we deduce the quasi-isomorphism
\begin{align*}
(\Lambda (V_{\B \SU(n)} \oplus V_{\SU(2n)}),\dif) \simeq (\Lambda \langle v_{2n+1}, \dots, v_{4n-1}\rangle,0).
\end{align*}
Clearly, just as in Case 1, this results from taking the quotient by the contractible algebra of all those elements of degree at most $2n$. Formality follows likewise.

\vspace{5mm}

\case{3} Basically the same arguments as in the first case apply to $\SO(2n)/\SO(n)$ with the double block standard inclusion. Here, however, we have to distinguish two cases.

First, we assume that $n$ is odd. The cohomology of $\B \SO(2n)$ is generated by elementary symmetric polynomials in the $t_i^2$ (except for the top degree one) where the $t_i$ generate $H^*(\B T)$ with $T\In \SO(2n)$ the maximal torus, together with the polynomial $t_1 t_2 \cdots t_n$ in degree $2n$ (compensating for the lack of $t_1^2 \cdots t_n^2$). Using the analogous description for $\B\SO(n)$ we find that the standard inclusion $\SO(n)\hto{} \SO(2n)$ injects all the rational homotopy groups of $\SO(n)$. They lie in degrees $3,7, \ldots, 2n-3$ and, in these degrees, they correspond bijectively to the rational homotopy of $\SO(2n)$, which are concentrated in degrees $3,7,\ldots,4n-5$ and $2n-1$. Interpreting these results for the double blockwise inclusion, we see that in these low degrees the induced map on rational homotopy groups is just multiplication with $2$ again. Thus, by taking the quotient by a contractible algebra, we find
\begin{align*}
(\Lambda (V_{\B \SO(n)} \oplus V_{\SO(2n)}),\dif)\simeq (\Lambda \langle v'_{2n-1}, v_{2n+1}, v_{2n+5}, \ldots, v_{4n-5}\rangle,0).
\end{align*}
Again, this homogeneous space rationally is a product of odd-dimensional spheres and formal, in particular.

\vspace{5mm}

Let us now assume that $n$ is even. With the analogous arguments we find that all the rational homotopy groups of $\SO(n)$ which are represented by symmetric polynomials in the $t_i^2$ are mapped bijectively to the corresponding ones in $\SO(2n)$. There remains the class $x_{n}$ represented by $t_1 \cdots t_{n/2}$ and the minimal model is
\begin{align*}
(\Lambda (V_{\B \SO(n)} \oplus V_{\SO(2n)}),\dif)\simeq (\Lambda \langle x_{n}, v'_{2n-1}, v_{2n-1}, v_{2n+3}, \ldots, v_{4n-5}\rangle,\dif)
\end{align*}
with $\dif(v'_{2n-1})=x_{n}^2$ and $\dif$ vanishing on all other generators. (The first property holds, since $t_1\cdots t_n$ maps to $t_1^2\cdots t^2_{n/2}$ under the morphism induced on the cohomologies of classifying spaces by diagonal inclusion.) Thus, the space rationally is the product of odd-dimensional spheres with exactly one even-dimensional sphere and again formal.

\vspace{5mm}

\case{4} We consider $\SO(4n)/\Sp(n)$. The morphism induced in the cohomology of classifying spaces $H^*(\B \SO(4n)) \to H^*(\B \Sp(n))$ is induced on formal roots by \[t_i\mapsto \tilde t_{i/2}, \quad t_{i+1} \mapsto -\tilde t_{i/2}\] for $i \equiv 0 \mod 2$ for \[H^*(\B T_{\SO(4n)})=\qq[t_1,\ldots, t_{2n}],\quad H^*(\B T_{\Sp(n)})=\qq[\tilde t_1,\ldots, \tilde t_{n}].\] In particular, this means that the $k$-th elementary symmetric polynomial in the $t_i^2$ is mapped to a non-trivial multiple of the $k$-th elementary polynomial in the $\tilde t_i^2$ modulo another symmetric polynomial generated by the first $k-1$ elementary symmetric polynomials. Alternatively: a power sum in the $t_i^2$ maps to twice a respective one in the $\tilde t_i^2$. By either perspective we derive that the inclusion induces an isomorphism on the first $n$ non-trivial rational homotopy groups. In other words, it follows that the map induced in the cohomology of classifying spaces is surjective, and that $H^*(\SO(4n)/\Sp(n))$ is an exterior algebra, which is formal.

\vspace{5mm}

\case{5} Let us now deal with $\Sp(n)/\SO(n)$ in a similar manner.
The inclusion of $\SO(n)\In \Sp(n)$ is induced by the componentwise inclusion $\rr\to\hh$. That is, on the torus we have the induced morphism
\[t_1\mapsto t_1,\; t_2\mapsto -t_1, \; t_3\mapsto t_2, t_4\mapsto -t_2,\; t_5\mapsto t_3, t_6\mapsto - t_3\ldots,\; t_n\mapsto 0\]
if $n$ is odd, and
\[t_1\mapsto t_1,\; t_2\mapsto -t_1, \; t_3\mapsto t_2, t_4\mapsto -t_2,\; t_5\mapsto t_3, t_6\mapsto - t_3\ldots,\; t_n\mapsto  -t_{ n/2 }\]
if $n$ is even.

Let us first assume that $n$ is odd. Then these considerations imply that the inclusion $\SO(n)\In \Sp(n)$ is rationally $4((n-1)/2)$-connected. Again, the quotient is just a free commutative graded algebra on odd generators. Indeed, we see that the power sums $\sum_{1\leq i\leq n} (t_i^2)^k$ (for $1\leq  k\leq n$) map to respective power sums $2\sum_{1\leq i\leq n/2 } (t_i^2)^k$.

If $n$ is even, then the same arguments imply that the quotient is just a product of an even-dimensional sphere (with volume form corresponding to the product $t_1\cdots t_{n/2}$, the $t_i$ coming from $H^*(\B\SO(n))$) and a free commutative algebra generated in odd degrees. In any case the space is formal.

\subsection{The space $\SU(n+m)/\SO(n)\times \SO(m)$ where $n$ and $m$ are odd}\label{subsec03}

Let us begin with an observation.
\begin{lemma}\label{lemfo}
Suppose that $G/H$ with inclusion $i\co H\to G$ is such that $\rk G-\rk H=k$ and $\dim \ker \pi^*(i)\otimes \qq=k$ (on dual rational homotopy groups), then $G/H$ is formal.
\end{lemma}
\begin{prf}
In this situation the induced morphism $\pi^*(\B G)\otimes \qq \to \pi^*(\B H)\otimes \qq$ has $k$-dimensional kernel as well. This implies that in the differential of the standard model of $G/H$ which we use, the differential vanishes on $k$ generators. Since $G/H$ has finite dimensional cohomology, we know that in the decomposition of Proposition~\ref{prop02} we have that $\dim (V')^\odd \geq (\dim V')^\even$. Counting generators it follows that equality holds in this inequality, and $(\Lambda V,\dif)$ is formal.
\end{prf}

We shall apply this lemma in this section and the subsequent ones.

We consider $\SU(2n+2m+2)/\SO(2n+1)\times \SO(2m+1)$.  We compute the Sullivan model of the space. (We may assume that the maximal torus of the first denominator factor is embedded into complex coordinates $1$ to $2n$, the one of the second factor into coordinates $2n+3$ to $2n+2m+2$.)
We obtain the Sullivan model
\begin{align*}
(\Lambda (V_{\B \SO(2n+1)}\oplus V_{\B \SO(2m+1)} \oplus V_{\SU(2n+2m+2)}),\dif)
\end{align*}
with the differential induced by
\begin{align*}
&t_1\mapsto t_1, t_2\mapsto -t_1, t_3\mapsto t_3, t_4\mapsto -t_3, \ldots,
\\&t_{2n-1}\mapsto t_{2n-1}, t_{2n}\mapsto t_{2n-1},
\\&t_{2n+1}\mapsto 0, t_{2n+2}\mapsto 0,
\\&t_{2n+3}\mapsto t_{2n+3}, t_{2n+4}\mapsto -t_{2n+3}, \ldots,
\\& t_{2n+2m+1}\mapsto t_{2n+2m+1}, t_{2n+2m+2}\mapsto - t_{2n+2m+1}
\end{align*}
where $t_1,t_3,t_5, \ldots, t_{2n-1}$ generate the cohomology of the torus of $\SO(2n+1)$ and $t_{2n+3}, t_{2n+5}, \ldots, t_{2n+2m+1}$ generate the cohomology of the torus of $\SO(2m+1)$.

Since the rational homotopy groups of $\B\SO(2n+1)$ respectively the ones of $\B\SO(2m+1)$ are concentrated in degrees divisible by four, so is their rational cohomology, and we derive that $\dif v=0$ for $v\in V_{\SU(2n+2m+2)}$ of degree congruent to $1$ modulo $4$. It follows that
\begin{align*}
\dif v_5=\dif v_9=\dif v_{11}=\dots=\dif v_{2(2n+2m+2)-3}=0
\end{align*}
The differential of the generator in top degree $2(2n+2m+2)-1$ corresponds to
\begin{align*}
t_1t_2\cdots t_{2n+2m+2}\mapsto \pm t_1^2t_3^2t_5^2\cdots t_{2n-1}^2\cdot 0\cdot 0 \cdot t_{2n+3}^2t_{2n+5}^2\cdots t_{2n+2m+1}^2
\end{align*}
Consequently, also its differential vanishes. Thus the model splits as
\begin{align*}
&(\Lambda (V_{\B \SO(2n+1)}\oplus V_{\B \SO(2m+1)} \oplus V_{\SU(2n+2m+2)}),\dif)
\\ \simeq & (\Lambda (V_{\B \SO(2n+1)}\oplus V_{\B \SO(2m+1)} \oplus \langle v_3, v_5, v_7, \ldots, v_{2(2n+2m+2)-5}\rangle),\dif)\\ & \otimes (\Lambda \langle v_{2(2n+2m+2)-3}, v_{2(2n+2m+2)-1}\rangle,0)
\end{align*}

The crucial observation we now make is that the first factor is exactly the Sullivan model one constructs for the homogeneous space $\SU(2n+2m)/(\Sp(n)\times \Sp(m))$. The latter space shares its maximal torus with the symmetric space $\SU(2n+2m)/\Sp(n+m)$; thus it is formal.

Alternatively, due to the observation above this model is of the form
\begin{align*}
&(\Lambda (V_{\B \SO(2n+1)}\oplus V_{\B \SO(2m+1)} \oplus V_{\SU(2n+2m+2)}),\dif)
\\ = & (\Lambda (V_{\B \SO(2n+1)}\oplus V_{\B \SO(2m+1)} \oplus \langle v_3, v_7, v_{11}, \ldots, v_{2(2n+2m+2)-5}\rangle),\dif)\\ & \otimes \Lambda \langle v_5, v_9, \ldots, v_{2(2n+2m+2)-3}, v_{2(2n+2m+2)-1}\rangle,0)
\end{align*}
So let us spell out and thereby illustrate an application of Lemma~\ref{lemfo} in this case: The first factor has as many generators as relations, i.e., it is positively elliptic and formal.
Hence the model of $\SU(2n+2m+2)/\SO(2n+1)\times \SO(2m+1)$ is the product of two formal Sullivan algebras and formal, consequently.

\subsection{The spaces $\SO(a+b+c)/\SO(a)\SO(b)\SO(c)$, where all three of the $a,b,c$ are odd and $\SO(a+b+c+d)/\SO(a)\SO(b)\SO(c)\SO(d)$, where three or four of the $a,b,c,d$ are odd}\label{subsec04}
Let us first deal with the first space. The inclusion of the stabiliser group is blockwise. The stabiliser group has corank $1$. Thus, in view of Lemma~\ref{lemfo}, it suffices to observe that
the top elementary symmetric polynomial in the $t_i^2$ from $\SO(a+b+c)$, namely $t_1^2\cdots t_{(a+b+c-1)/2}^2$, maps to zero under the map induced on classifying spaces by the inclusion of the stabiliser. Consequently, the remaining cohomology generators of the numerator need to map to a regular sequence in $H^*(\B(\SO(a)\SO(b)\SO(c))$, since cohomology is finite-dimensional (see Proposition~\ref{prop02}). It follows that the minimal model for the homogeneous space splits as this positively elliptic space tensored with the free commutative algebra generated by the generator of $\SO(a+b+c)$ corresponding to $t_1^2\cdots t_{(a+b+c-1)/2}^2$.

\vspace{5mm}

Let us now consider the second space and first assume that all of $a,b,c,d$ are odd. Then the stabiliser has corank two and, similar to the last case, now the top two rational homotopy groups generated by $t_1^2\cdots t_{(a+b+c-1)/2}^2$ and by the previous elementary symmetric polynomial in the $t_i^2$ of $\SO(a+b+c+d)$ map to zero. For this it simply remains to observe that also any monomial of the form $t_1^2\cdots \hat t_i^2 \cdots t_{(a+b+c-1)/2}^2$ (where $\hat \cdot$ denotes omission) maps to $0$ under restriction as well, as, by corank, there is necessarily one $t_j$ in this product which goes to zero.

Thus the homogeneous space rationally is the product of a positively elliptic space and the free algebra generated by two elements. In particular, due to Lemma~\ref{lemfo} it is formal.

Suppose now that $a,b,c$ are odd, $d$ is even. Then the corank is $1$, the top rational homotopy group restricts to zero; again we have a similar splitting and formality due to Lemma~\ref{lemfo}.

\subsection{The spaces $\E_7/\F_4$, $\E_7/\Sp(4)$, $\E_7/\SO(8)$}\label{subsec05}
In order to prove the formality of $\E_7/\F_4$ we shall use merely two pieces of information:
\begin{enumerate}
\item the rational homotopy groups of $\E_7$ and $\F_4$,
\item the fact that the inclusion of simple simply-connected compact Lie groups induces an isomorphism in third rational cohomology, i.e., is rationally $4$-connected.
\end{enumerate}
In Section~\ref{sec05} when we prove the equivariant formality of the isotropy action of these spaces, we will need a little more information which we will cite from the literature.

Let us elaborate briefly on the information we use.
We cite the rational homotopy groups of the exceptional Lie groups from \cite[p.~347]{MT91}.}
The rational homotopy groups of $\E_7$ are concentrated in degrees
\begin{align*}
3,11,15,19,23,27,35,
\end{align*}
and each group is one-dimensional.
For $\F_4$ we have one-dimensional rational homotopy groups in degrees
\begin{align*}
3,11,15,23.
\end{align*}

In our situation the third cohomology group is spanned the form \linebreak[4]$g(x,y,z)=B(x,[y,z])$ where $B$ is the Killing form. The pullback of a Killing form is an integral non-trivial multiple of the Killing form, since so is every invariant non-trivial bilinear form.
We shall essentially draw on the following observation for a pure elliptic Sullivan algebra $(\Lambda V,\dif)$ with basis $(v_i)_i$ of $V^\even$: since cohomology is finite dimensional, any cohomology class $[v_i]$ is nilpotent. That is, in particular, for every $v_i$ there exists some $x_i\in \Lambda V$ with $\dif x_i=v_i^{k_i}$ modulo the ideal generated by the $v_j\neq v_i$ for some $k_i\in \nn_0$.
Moreover, for minimal such $k_i$ (which we may assume without loss of generality), the $x_i$ necessarily lie in $V^\odd$ by word-length considerations, since $\dif$ is a derivation.
In our case, the degrees of the $v_i$ and $x_i$ are determined by the specific Lie groups. Hence this observation will impose severe restrictions on the structure of the differential. More precisely, we shall see that (up to isomorphism) there are $\rk G-\rk H$ generators $x_i$ of $H^*(G)$, i.e., rational homotopy group generators of $G$, with $\dif x_i=0$. An application of Lemma~\ref{lemfo} then yields formality.

\vspace{5mm}

We are now ready to start computing a Sullivan model of $\E_7/\F_4$. It is given by
\begin{align*}
(\Lambda (V_{\B \F_4} \oplus V_{\E_7}),\dif)
\end{align*}
and the differential vanishes on the first summand. Write
\begin{align*}
V_{\B \F_4}=\langle v_{4},v_{12},v_{16},v_{24}\rangle.
\end{align*}
The cohomology is finite-dimensional.
Denote by $x_i$ a generator of the minimal model of $\E_7$ of degree $i$ and by $v_i$ such a generator of the minimal model of $\B \F_4$. It follows by degree considerations that $\dif x_3=v_4$, $\dif z_{11}=k_{11}v_{12}$ modulo the ideal generated by $v_4$, $I(v_4)$, for some $k_{11}\in \qq$. We invoke the principle discussed above now and derive that, without loss of generality, $\dif x_{23}=v_{24}$---by degree considerations there is no element mapping to a power of $v_{24}$ other than in degree~$23$; moreover, by word-length, $x_{23}$ is the only element which can map to $v_{24}$. The same argument yields that, without loss of generality, $\dif x_{15}=v_{16}$. Applying this reasoning to $v_{12}$ lets us deduce that either it is hit by $x_{11}$, its square is hit by $x_{23}$, or its cube by $x_{35}$. Since we already had $\dif x_{23}=v_{24}$, these two elements contract in the minimal model, and using the reasoning on the minimal model then shows that, without loss of generality, $\dif x_{11}=v_{12}$ or $\dif x_{35}=v_{12}^3$.

Let us now distinguish these two cases: In the first case a minimal model of the homogeneous space is $(\Lambda \langle x_{19}, x_{27}, x_{35}\rangle,0)$. In the second case it is  $(\Lambda \langle v_{12}, x_{11}, x_{19}, x_{27}\rangle, 0) \otimes (\Lambda \langle v_{12}, x_{35}\rangle, x_{35}\mapsto v_{12}^3)$. Both these models are formal; in particular, they show that once again Lemma~\ref{lemfo} applies.

\vspace{5mm}

As for the formality of $\E_7/\Sp(4)$ and $\E_7/\SO(8)$, it suffices to observe that both subgroups $\Sp(4)$ and $\SO(8)$ share one of their respective maximal tori with the subgroup~$\F_4 \In \E_7$ considered above. To see this for $\Sp(4)$, note that both subgroups $\Sp(4)$ and~$\F_4$ after conjugation, contain one and the same subgroup~$\Sp(3)\Sp(1)$ of full rank~\cite[Table~1]{Kol09}. Similarly, both the subgroups $\Sp(4)$ and $\SO(8)$, which are contained in~$\SU(8) \In \E_7$, after conjugation, contain one and the same subgroup~$\U(4)$ of full rank; see~\cite[Table~4]{BG08}.

\section{A criterion for equivariant formality}\label{sec04}

The results in this section can probably as well be deduced from the cited results in the literature. However, this is not completely obvious, as we use different perspectives and language which suggests the following direct and self-contained outline as a service to the reader. Notably, our methods reveal a beautiful parallelism of the criteria used for formality and for equivariant formality with very subtle differences---see Remark~\ref{remsim}.

We present a characterisation of the equivariant cohomology of the isotropy action which builds on the formality of $G/K$ and provides an additional condition to check (cf.~\cite[Theorem 1.4]{CF16}). (We observe that the arguments on top of \cite[Page~479]{CF16} and in the proof of \cite[Lemma 3.9]{CF16} seem to suggest a similar splitting of models, a trivialisation of differentials as we shall provide in the proofs below.)
We then invite the reader to compare it to our closely related formality criterion in Proposition~\ref{prop02} in view of Remark \ref{remsim}.

Note that $\pi^*(\B G)=\Hom(\pi_*(\B G),\qq)$ denote dual rational homotopy groups. The first inclusion expresses the fact that rational cohomology of $\B G$ is a polynomial algebra generated by spherical cohomology classes, i.e., by dual homotopy groups. The following theorem is included in~\cite[Proposition~3.10]{CF16}.

\begin{theo}\label{theo01}
Let $K \curvearrowright G/K$ be an isotropy action such that $G/K$ is formal. Suppose further that
\begin{align*}
\dim \ker\big(\pi^*(\B G)\hto{} H^*(\B G)\to H^*(\B K)\big)=\rk G-\rk K
\end{align*}
Then the isotropy action is equivariantly formal.
\end{theo}
\begin{prf}
Under the given assumptions we show that the induced map on rational cohomology of the fibre inclusion of the fibration
\begin{align*}
G/K \hto{} \E K \times_K G/K \to \B K
\end{align*}
is surjective. We form a model of the total space as a relative model of this fibration, i.e., actually as a biquotient model~\cite{Kap}, \cite[Chapter 3.4.2]{FOT08}.
This yields the Sullivan algebra
\begin{align*}
(H^*(\B K) \otimes H^*(\B K) \otimes H^*(G),\dif)
\end{align*}
encoding the rational homotopy type of the Borel construction, with its cohomology being equivariant cohomology. The differential is given by
\begin{align}\label{eqndif}
\dif|_{H^*(\B K) \otimes H^*(\B K)}=0, \quad \dif(x)=-H^*(\B \phi_1) (x^{(+1)}) + H^*(\B \phi_2) (x^{(+1)})
\end{align}
where $\phi_i$ are formal copies of $\phi$, just with respect to the two different factors $H^*(\B K)$, and extended as a derivation where $x$ is a spherical cohomology class, i.e., an algebra generator of $H^*(G)$, and $\phi\co K\to G$ denotes the inclusion. Here, $x^{(+1)}$ again denotes the suspension, i.e., a degree shift by $+1$.

Recall that the algebra generators of $H^*(\B G)$ correspond to its rational homotopy groups; the same is true for $G$, and elements are just related by a degree shift. That is, the condition from the assertion guarantees that also in the model of the Borel construction, possibly up to isomorphism, there are $(\rk G-\rk K)$ generators of $H^*(G)$ mapping to zero under the differential.

Now by construction of the model of the total space as a (minimal) relative Sullivan model, the fibre inclusion is modelled by the projection
\begin{align}\label{eqnpro}
&(\Lambda (H^*(\B K) \otimes H^*(\B K) \otimes H^*(G)),\dif)\\ \nonumber\to~~ & (\Lambda (H^*(\B K) \otimes H^*(\B K) \otimes H^*(G)),\dif) \otimes_{H^*(\B K)} \qq
\end{align}
(where we divide by the base $H^*(\B K)$ encoding the action).

Recall the classical observation that, by construction, both the Sullivan model of the Borel construction and the model of $G/K$ are pure. Moreover, if such a space is formal, then a model of it is necessarily \emph{isomorphic} to the tensor product of one of a rationally elliptic space of positive Euler characteristic and a free algebra generated in odd degrees (see Proposition~\ref{prop02}).

Hence, since $G/K$ is formal, the cohomology of this positively elliptic factor is generated by its $H^*(\B K)$. Similarly, we observe that we have a subalgebra of equivariant cohomology generated by $H^*(\B K)\otimes H^*(\B K)$. The first factor comes from the action, the second one from the model of $G/K$. Hence the second factor maps onto its analogue $H^*(\B K)$ in the fibre model of~$G/K$. Thus, the fibre projection surjects onto the positively elliptic cohomology. It remains to see that it also surjects onto the free part generated in odd degrees.

By construction, this free factor for the model of the fibre is generated by $(\rk G-\rk K)$ odd-degree elements. The crucial observation which will finish the proof is that the \emph{isomorphism} which yielded the product splitting is actually the \emph{identity}.
Above we observed that there are $\rk G-\rk K$ generators which map to zero under the differential of the model of $G/K$. The fibre projection is modelled in \eqref{eqnpro}, i.e., it is the identity on $H^*(G)$ and its $H^*(\B K)$ and projects away the action factor $H^*(\B K)$. Now we invoke the special nature of the action, i.e., our concrete knowledge of the model of the Borel construction \eqref{eqndif}. Hence, the perturbation of the differential encoding the action is just the negative of the morphism induced by the inclusion $K\hto{} G$ encoding the homogeneous space. It follows that also in the model of the Borel construction the aforementioned $\rk G-\rk K$ generators are mapped to zero under its differential. Since the models are pure, it follows that the projection is surjective on the cohomology generated by these odd-degree elements.

Summarizing, we have seen the following: Since one $H^*(\B K)$ from the Borel construction surjects onto the $H^*(\B K)$ of the fibre (via the identity), it follows that the fibre projection is surjective on the cohomology generated by even degree elements. It is also surjective on the cohomology generated by odd-degree elements (and hence  surjective in total), since
\begin{align*}
(H^*(\B K) \otimes H^*(\B K) \otimes H^*(G),\dif)
\end{align*}
splits as a product of an algebra with cohomology generated by $H^*(\B K)\otimes H^*(\B K)$ and a free algebra generated by exactly the same $(\rk G-\rk K)$ generators as we find in the analogous splitting on 
$(H^*(\B K) \otimes H^*(\B K) \otimes H^*(G),\dif) \otimes_{H^*(\B K)} \qq$.

Consequently, in total, the fibre projection induces a surjective morphism on cohomology. Hence the isotropy action is equivariantly formal.
\end{prf}

Let us refine this criterion a little more in the form of a corollary. We mainly want to apply this modification in the situation when an element possibly does not map to zero but to a contractible algebra, which essentially does not make any difference. For this we consider the model $(\Lambda V,\bar\dif)=(H^*(\B K)\otimes H^*(G),\dif)$ of $G/K$. By construction, an element $v\in V^\odd$ represents a rational homotopy group of $G$. The differential $\bar \dif v$ is exactly the map in the assertion of the theorem up to the usual degree shift.
This connects the theorem to its corollary.
Theorem~\ref{theo01} thus is extended by
\begin{cor}\label{cor01}
Let $(G,K)$ be a pair such that $G/K$ is formal, and denote by $(\Lambda V,\bar\dif)$ the standard Sullivan model of $G/K$. Suppose there exists a homogeneous subspace $W\In V^\odd$ with
\begin{align*}
\dim W=\rk G-\rk K
\end{align*}
and with a complement $W'\In V^\odd$ of $W$ satisfying the property that
\begin{align*}
\im \bar\dif|_W\In \Lambda(\im \bar \dif|_{W'})
\end{align*}
Then the isotropy action is equivariantly formal.
\end{cor}
\begin{prf}
We adapt the proof of the theorem in the following way. We have already recalled how the considered map translates to the differential on $\Lambda V$. Thus the only difference between corollary and theorem is that in the corollary the map/the differential restricted to $V^\odd$ does not necessarily map to zero, but to a subspace already in the image of other elements. We use this property to deform the model in such a way that we reveal an actual subspace in the kernel. Consequently, the theorem applies literally.

We extend the model $(\Lambda V,\bar \dif)$ to the model \linebreak[4]$(H^*(\B K)\otimes H^*(\B K)\otimes H^*(G),\dif)$ of the Borel construction in the usual way. By assumption, and by the symmetry of the differential on the Borel construction with respect to the two $H^*(\B K)$ factors (up to multiplication with $-1$), we have that
$\dif v=\bar \dif^+ v - \bar \dif^- v$
for $v\in V^\odd=\pi^*(G)$, where $\bar \dif^+$ and $\bar \dif^-$ denote formal copies of $\bar \dif$ just with respect to the two different $H^*(\B K)$ factors. We choose a homogeneous basis of $W$. Hence, for each of its elements $v$, we have that
\begin{align*}
\bar \dif v=\sum_i \prod_j \bar \dif w_{i,j}
\end{align*}
for some $w_{i,j} \in W'$, and that
\begin{align*}
\dif v=\sum_i \prod_j \bar \dif^+ w_{i,j} -\sum_i \prod_j \bar \dif^- w_{i,j}
\end{align*}
We claim that this expression is already the image under $\dif$ of an element $v'$ in the ideal generated by $W'$. For this it suffices to observe that in a Sullivan algebra an element $a\cdot b$ is homologous to $a'\cdot b'$ given that $\dif a=\dif a'=\dif b=\dif b'=0$, that $a$ is homologous to $a'$, and that $b$ is homologous to $b'$. In fact, let $\dif x=a-a'$ and $\dif y=b-b'$, then (we can assume that $\deg a$, $\deg b$, $\deg a'$, $\deg b'$ are even)
\begin{align}\label{eqnnnn}
\dif(xb+a'y) =ab-a'b'
\end{align}
Indeed, now suppose that $\bar \dif v$ lies in the ideal generated by $\bar \dif W'$, i.e., up to linearity say $\bar \dif v=b\cdot a$ with $a\in \im \bar \dif|_{W'}$. Then $a-a'\in \im \dif|_{W'}$ and  $\dif v=ba-b'a'$. At this point it becomes clear why we need to strengthen our assumption compared to the one in Proposition \ref{prop02}. Without further knowledge, we cannot deduce that also $\dif v$ lies in the ideal generated by $W'$ such that we could perform a change of basis and assume that $\dif v=0$. Now, however, by assumption, also $b\in  \Lambda (\im \bar \dif|_{W'})$. Via Equation \eqref{eqnnnn} we can run an induction over the number of factors of the elements $b$ (starting from monomials with at most two factors and passing to elements of more factors) in order to see that by the induction assumption it is safe to assume that also $b-b'$ is in the image of $\dif$. Thus, in total, again via Equation \eqref{eqnnnn} we can (finish the induction and) deduce that $ba-b'a'\in \im \dif$. Thus we can assume that $\dif v=\dif v'$ with $v'\in I(W')$ in the ideal generated by $W'$. Hence we can change the basis by the automorphism induced by $v\mapsto \tilde v+v'$ where $\tilde v:=v-v'$ and assume that $\dif v=0$ by replacing $v$ with $\tilde v$. We continue to do so for all the basis elements $v$. Hence we may apply the theorem. This finishes the proof.
%
%

%
%
%
%
\end{prf}

\begin{rem}\label{remsim}
It is a beautiful observation that the difference between Proposition~\ref{prop02} and Corollary~\ref{cor01}, i.e., between our sufficient criteria for formality respectively for equivariant formality is that in the first case it is sufficient that a $\rk G-\rk K$ dimensional subspace maps into the \emph{ideal} generated by the image of a complement under the differential (then a similar automorphism of the model shows that the differential becomes trivial on a $\rk G-\rk K$ dimensional subspace), in the second case it needs to map to the \emph{subalgebra} generated by the image of the differential of a complement.
\end{rem}

In the following corollary we apply Corollary~\ref{cor01} in the case when in the model of $G/K$ the subalgebra $H^*(\B K)$ is contractible. Hence $\rk G- \rk H$ generators of $V$ (due to the pureness of the model) necessarily map into this contractible algebra. Extending these generators by a basis of $H^*(\B K)$ allows us to apply the previous corollary.
\begin{cor}\label{cor02}
If $H^*(G/K)$ is a free commutative algebra generated in odd degrees, then the isotropy action is equivariantly formal.
\end{cor}
\begin{prf}
Such a space is intrinsically formal. 
Since the model is pure, the differential of the odd-degree generators maps to $H^*(\B K)$. Since cohomology is generated in odd degrees, it follows that there exists a minimal $W'\In \pi^*(G)$ with $\im \dif|_{W'}=\pi^*(\B K)$ admitting a homogeneous complement $W$ with $\dif W\in \Lambda (\pi^*(\B K))=H^*(\B K)$ and satisfying $\dim W=\rk G-\rk K$. Hence Corollary~\ref{cor01} applies to $W$, $W'$.
\end{prf}
This originally was expressed by Shiga in the form that the isotropy action is equivariantly formal if $H^*(G) \to H^*(K)$ is surjective.

\vspace{5mm}

Our subsequent discussion of equivariant formality will draw on the criterion as presented in the corollary together with the next simple observation.
The following proposition is also well-known (see \cite[Proposition C.26, p.~207]{GGK02})
and follows basically from the general form of the following commutative diagram (which we already specialise to our concrete case) with horizontal and vertical fibrations.
\begin{align*}
\xymatrix{
\ast \ar@{^{(}->}[r]\ar@{^{(}->}[d] & K/T \ar[r]\ar@{^{(}->}[d] & K/T\ar@{^{(}->}[d]\\
G/K\ar@{^{(}->}[r]\ar[d] &\E T \times_T G/K\ar[r]\ar[d] & \B T \ar[d]\\
G/K\ar@{^{(}->}[r] &\E K \times_K G/K \ar[r] & \B K
}
\end{align*}

\begin{prop}\label{prop01}
If a compact connected Lie group $G$ acts on a compact manifold $M$, then its action is equivariantly formal if and only if so is the induced action of its maximal torus $T$.
\end{prop}

Again, we obtain an analogue of Corollary~\ref{cor03} (see \cite[Theorem 1.1]{Car19}, \cite[Proposition~3.5(iii)]{GN16}).

\begin{cor}\label{cor04}
Suppose that the maximal tori $T_K\In K$ and $T_{K'}\In K'$ are conjugate in $G$. Then the isotropy action $K\curvearrowright G/K$ is equivariantly formal if and only if the isotropy action $K'\curvearrowright G/K'$ is equivariantly formal.
\end{cor}
\begin{prf}
We first recall that $K\curvearrowright G/K$ is equivariantly formal if and only if so is $T_K\curvearrowright G/K$. Now we use the symmetry between left action and right action in order to deduce that the right action of $K$ on the homogeneous space $T_K\backslash G$ is equivariantly formal if and only if so is the right action $T_K\backslash G \curvearrowleft T_K$. Indeed, we have the following equivalences
\begin{align*}
&\textrm{$K$ acts equivariantly formally on $G/K$ from the left.}\\
\iff&\textrm{$H^*(K\setminus G/K)$ surjects onto $H^*(G/K)$.}\\
\iff&\textrm{$H^*(K\setminus G/K)$ surjects onto $H^*(K\setminus G)$.}\\
\iff&\textrm{$H^*(T_K\setminus G/K)$ surjects onto $H^*(T_K\setminus G)$.}\\
\iff&\textrm{$K$ acts equivariantly formally from the right on $T_K\setminus G$.}\\
\iff&\textrm{$T_K$ acts equivariantly formally from the right on $T_K\setminus G$.}
\end{align*}
In the second equivalence we just switched sides, i.e., we use the same morphism and just exchange the $H^*(\B K)$-factors in the model. This only affects the differential by $-1$ and has no effect on surjectivity.

Only the third equivalence needs some more arguments. We have the obvious commutative diagram
\begin{align*}
\xymatrix
{1 \ar[r]\ar[d] & K/T_K \ar[d]_<<<<<<{} \ar[r]^<<<<<<{} & K/T_K \ar[d]\\
  K\ar[r]\ar[d]& T_K\setminus G \ar[r]\ar[d]_<<<<<<{}  & T_K\setminus G/K \ar[d] \\
  K \ar[r]&K\setminus G \ar[r]& K\setminus G/K
}
\end{align*}
in which all maps are fibrations or fibre inclusions. All the three lower vertical morphisms induce injective morphisms in rational cohomology, since as vector spaces the cohomologies of the spaces in the middle row split as tensor products of the cohomologies of the  corresponding spaces in the top and the bottom row---which clearly follows from a classical spectral sequence arguments using that $\rk T_K=\rk K$ whence $H^\odd(K/T_K)=0$. It follows that the morphism induced by the middle horizontal fibration
\begin{align*}
& H^*(T_K\setminus G/K)\cong H^*(K\setminus G/K)\otimes  H^*(K/T_K) \\ \to &H^*(T_K\setminus G)\cong H^*(K\setminus G) \otimes H^*(K/T_K)
\end{align*}
is surjective if and only if so is the one in the bottom row $H^*(K\setminus G/K)\to H^*(K\setminus G)$.

Finally, since $T_K$ and $T_{K'}$ are conjugate, the action $T_K \setminus G \curvearrowleft T_K$ is equivalent to $ T_{K'}\backslash G \curvearrowleft T_{K'}$. Then we may go all the way backwards through the equivalences and see that the equivariant formality of the action is equivalent to the one of $K'\curvearrowright G/K'$.
\end{prf}

\section{Equivariant formality of ($\zz_2\oplus \zz_2$)-symmetric spaces}\label{sec05}

We now reduce to the case of a simple numerator group $G$ by first dealing with the Cases \eqref{DecItemTorus},  \eqref{DecItemTwo}, \eqref{DecItemTwoZ}, \eqref{DecItemFour} from Proposition~\ref{prop03}, where $G$ is non-simple. For these it remains to make the following observations: In Cases \eqref{DecItemTorus},  \eqref{DecItemTwo}, \eqref{DecItemFour} both minimal model and cohomology are given by an exterior algebra, and we can cite Corollary~\ref{cor02}. In any of the four cases, as we observed in Section~\ref{secfrt}, the model is of the form $(\Lambda \langle x_1,\ldots , x_l\rangle,0) \otimes (A,\dif)$ (with $\deg x_i$ odd, and with $A$ only possibly non-trivial in \eqref{DecItemTwoZ}).

If $A$ is not trivial, the algebra $(A,\dif)$ is a model of a symmetric space, and equivariant formality in our case follows from the equivariant formality of the isotropy action on symmetric spaces, as the isotropy action restricts accordingly.

\vspace{5mm}

Let us recall a few easy observations:  Equal rank homogeneous spaces have equivariantly formal isotropy actions (as the $\E_2$-page of the Leray--Serre spectral sequence associated to the Borel fibration is concentrated in even total degrees only; even more generally, since homogeneous spaces satisfy the Halperin conjecture). So do symmetric spaces (see \cite{Goe12}). Combining all this with Proposition~\ref{prop01}, it suffices, in view of Corollary~\ref{cor04}, to focus on the very same list of remaining cases we gave in Subsection~\ref{subsec01} for the formality issue.

Having proved the formality of $(\zz_2\oplus \zz_2)$-symmetric spaces, in view of Theorem~\ref{theo01} it remains to check the induced map on cohomology algebra generators
\begin{align*}
\pi^*(\B G)\otimes \qq\hto{} H^*(\B G)\to H^*(\B K)
\end{align*}
and the dimension of its kernel. Moreover, we may again restrict to only dealing with one representative out of those spaces whose stabilizer groups share maximal tori.

Moreover, we may refer to Subsections~\ref{subsec02}, \ref{subsec03}, \ref{subsec04}, \ref{subsec05},
where we already did this analysis. Basically, we can finish the proof by stating that in any of these cases what we checked in order to apply Lemma~\ref{lemfo} is just good enough to directly allow for an application of Corollary~\ref{cor01}.

Let us quickly recall the crucial observations from there:

\begin{itemize}
\item The spaces $\Sp(2n)/\Sp(n)$, $\SU(2n)/\SU(n)$, $\SO(2n)/\SO(n)$, \linebreak[4]$\SO(4n)/\Sp(n)$,  $\Sp(n)/\SO(n)$ all came from block inclusions (see Section~\ref{subsec02}). In particular, we observed that in nearly all cases these spaces had a free commutative cohomology algebra (generated in odd degrees). Hence, due to Corollary~\ref{cor02} the isotropy action is equivariantly formal.

    It only remains to consider the two exceptions, namely, the spaces $\SO(2n)/\SO(n)$ from Case 3 of the formality reasoning and $\Sp(n)/\SO(n)$ from Case 5, both with $n$ even. In these cases the cohomology of $G/K$ splits as the product of the cohomology of an even-dimensional sphere and odd-dimensional spheres. More precisely, we identified the model as
\begin{align*}
(\Lambda (V_{\B \SO(n)} \oplus V_{\SO(2n)}),\dif)\simeq (\Lambda \langle x_{n}, v'_{2n-1}, v_{2n-1}, v_{2n+3}, \ldots, v_{4n-5}\rangle,\dif)
\end{align*}
with $\dif(v'_{2n-1})=x_n^2$ and $\dif$ vanishing on the other generators.
Indeed, as for the standard model we observed that the elements $v_3, \ldots, v_{2n-5}$ map into and contract $\qq[x'_4, \ldots, x'_{2n-4}]$ under $\dif$, where $x'_i$ denotes the elementary symmetric polynomial in the squares $t_j^2$ of the formal roots $t_j$ in degree $i$. Since the $v_{2n-1},\ldots, v_{4n-5}$ are mapped to classes corresponding to invariant polynomials in the \emph{squares} of the formal roots, namely in the $t_i^2$, it follows that they are mapped into
\begin{align*}
\Lambda \langle x_n^2,x_4', \ldots, x_{2n-4}'\rangle=\Lambda \langle \dif v'_{2n-1}, \dif v_3, \ldots, \dif v_{2n-5}\rangle
\end{align*}
Hence, in order to deduce equivariant formality, we can apply Corollary~\ref{cor01} with $W'=\langle v'_{2n-1}, v_3, \ldots, v_{2n-5}\rangle$ and with \linebreak[4]$W=\langle v_{2n-1},\ldots, v_{4n-5}\rangle$ satisfying $\dim W=\rk G-\rk K=n/2$.

\vspace{5mm}

The situation for $\Sp(n)/\SO(n)$ with $n$ even is similar. Mimicking the notation from the last case, the model is generated by the elements $x_4',x_8',\ldots, x_{2n-4}',x_n$ and $v_3,v_7, \ldots, v_{4n-1}$. We have seen that the differential maps $v_{2n+3},\ldots, v_{4n-1}$ to zero.
These span a vector space of dimension $\rk G-\rk K=n/2$. Thus via Theorem~\ref{theo01} we directly
deduce equivariant formality.

\item In the case of the space $\SU(2n+2m+2)/\SO(2n+1)\times \SO(2m+1)$, we observed in Section~\ref{subsec03} that its standard model is of the form
     \begin{align*}
     &(\Lambda (V_{\B \SO(2n+1)}\oplus V_{\B \SO(2m+1)} \oplus V_{\SU(2n+2m+2)}),\dif)
     \\ = & (\Lambda (V_{\B \SO(2n+1)}\oplus V_{\B \SO(2m+1)} \oplus \langle v_3, v_7, v_{11}, \ldots, v_{2(2n+2m+2)-5}\rangle),\dif)\\ & \otimes \Lambda \langle v_5, v_9, \ldots, v_{2(2n+2m+2)-3}, v_{2(2n+2m+2)-1}\rangle,0)
     \end{align*}
     Thus the differential vanishes on the $v_i$ of degree congruent to $1$ modulo $4$ and additionally on the top degree $v_{2(2n+2m+2)-1}$; in total, on
     \begin{align*}
     ((2(2n&+2m+2)-3)-1)/4+1\\
     =~& n+m+1\\
     =~&(2n+2m+1)-(2n)/2-(2m)/2\\
     =~&\rk G-\rk K
     \end{align*}
     generators. Thus we apply Theorem~\ref{theo01} to deduce the result.
\item The spaces $\SO(a+b+c)/\SO(a)\SO(b)\SO(c)$ and $\SO(a+b+c+d)/\SO(a)\SO(b)\SO(c)\SO(d)$ (with either three or all four of the $a,b,c,d$ odd) result from standard block inclusion. In Section~\ref{subsec04} we observed that in the models of these three cases, the top, top two, and top odd-degree generators (respectively) map to zero under the differential, and accordingly, the coranks of the groups are $1$, $2$, $1$. Thus we invoke Theorem~\ref{theo01} to deduce equivariant formality.
\item The space $\E_7/\F_4$. (In Section \ref{subsec05} we observed that $\F_4$ shares a maximal torus with $\Sp(4)$ and $\SO(8)$; hence the subsequent arguments also cover the spaces $\E_7/\Sp(4)$ and $\E_7/\SO(8)$.)  The cohomology of $\E_7$ is generated in degrees $3, 11, 15, 19, 23, 27, 35$, the one of $\F_4$ in degrees $3, 11, 15, 23$. Extending the discussion of this case from above, we observe that either the generators in degree $19$, $27$ and $35$ map into the ideal generated by a contractible algebra (generated by those in degree $3$ and $4$, $11$ and $12$, $15$ and $16$, $23$ and $24$), or there are cohomology generators in degrees $11$, $12$.

    From \cite[Table 1, p.~294]{Pic16} we have the Poincar\'e polynomials of the symmetric spaces $\E_7/\E_6\U(1)$ and $\E_6/\F_4$. We draw two observations from these. First, $\pi^{12}(\B \E_7)\to \pi^{12}(\B \E_6)$ is an isomorphism, and second, so is $\pi^{12}(\B \E_6)\to \pi^{12}(\B \F_4)$. This lets us deduce that the differential on the underlying vector space of the model of $\E_7/\F_4$ is an isomorphism from degree $11$ to degree $12$, and we are in the first case.
It follows that the cohomology algebra is a free algebra generated in odd degrees, and we conclude by applying Corollary~\ref{cor02}.
\end{itemize}

\section{A short alternative proof of the equivariant formality of the isotropy action on symmetric spaces}\label{sec06}

We provide a short argument for the equivariant formality of symmetric spaces building upon their classification together with the classical fact that they are (geometrically) formal.

It is a short and classical argument to prove that compact Lie groups, considered as symmetric spaces $(G\times G)/\Delta G$ (with the diagonal inclusion $\Delta G$) have equivariantly formal isotropy actions. In this case, the quotient $G\cong (G\times G)/\Delta G$ is obviously formal, and, rationally, the fibre projection in the Borel fibration is just
\begin{align*}
(H^*(\B G)\otimes H^*(\B G) \otimes H^*(G)\otimes H^*(G),\dif) \to H^*(G),
\end{align*}
where the differential is such that the cohomology of the anti-diagonal $\{(g,-g) \mid g\in G\}$ in $G\times G$ projects surjectively onto $H^*(G)$.

\vspace{5mm}

Since product actions are equivariantly formal if and only if their factors are, and since homogeneous spaces $G/K$ with $\rk G=\rk K$ have equivariantly formal isotropy actions, it remains to prove the equivariant formality of the isotropy actions of those irreducible symmetric spaces with $\rk K<\rk G$. These then come out of a short list:
\begin{itemize}
\item The spaces $\SU(n)/\SO(n)$. Let us make a distinction by the parity of $n$. Suppose first that $n$ is odd. The algebra $H^*(\B\SU(n))$ is concentrated in even degrees $4, 6, \ldots, 2n$ whereas $H^*(\B \SO(n))$ ($n$ odd) is concentrated in degree divisible by four, namely $4,8, \ldots, 2n-2$. Hence, merely for degree reasons, exactly half, i.e., $(n-1)/2$, of the generators of $H^*(\B \SU(n))$, namely the ones in degrees $6,10,\ldots, 2n-4, 2n$ map to zero under the morphism induced by the group inclusion. Since the corank of $\SO(n)$ in $\SU(n)$ is exactly $(n-1)-(n-1)/2=(n-1)/2$ we deduce that the isotropy action is equivariantly formal by Theorem~\ref{theo01}.

Suppose now that $n$ is even. We repeat the argument from $n$ odd observing that now $H^*(\B \SO(n))$ is generated in degrees $4,8, \ldots, 2n-4$ with an additional generator in degree $n$. If $n$ is divisible by four, exactly the same arguments as above apply: For degree reasons the generators in degrees $6,10, \ldots, 2n-2$ map to zero; these are $n/2-1=(n-1)-n/2=\rk \SU(n)-\rk \SO(n)$.

If $n \equiv 2 \mod 4$, we need to observe that the additional generator in degree $n$ of $H^*(\B \SO(n))$, namely the one corresponding to $t_1\cdots t_{n/2}$, might potentially be hit by the generator of $H^*(\B \SU(n))$ in this degree, namely the element corresponding to the $n/2$-th elementary symmetric polynomial $t_1\cdots t_{n/2}+\dots + t_{n/2+1}\cdots t_n$, in the formal roots, i.e., in the algebra generators of the cohomology of the classifying space of its maximal torus. However, recall that $t_1\mapsto t_1$, $t_2\mapsto -t_1$, $t_3\mapsto t_2$, $t_4\mapsto -t_2$, etc. Consequently, every summand of $t_1\cdots t_{n/2}+\dots + t_{n/2+1}\cdots t_n$ which maps to $t_1\cdots t_{n/2}$ drags along another summand which maps to its negative. So also here the induced morphism vanishes.

Moreover, we observe that the Chern classes of $H^*(\B\SU(n))$ of degrees $n+4, n+8,\ldots, 2n-2$ map to zero, which, for degree reasons is equivalent to saying that they do not map into the ideal generated by $t_1\cdots t_{n/2}$. This can be deduced as follows: Any element in this ideal (not in degrees $n,2n$) consists of summands which contain third powers of some $t_i$. In contrast, the image of any Chern class can only contain squares of the $t_i$ at most. (In degree $2n$ we observe that the class corresponding to $t_1\cdots t_n$ does map to the one corresponding to $-(t_1\cdots t_{n/2})^2$.)
In summary, the Chern classes in degrees $6,\ldots,n, n+4, \ldots, 2n-2$ map to zero. These are $n/2-1$, and $\rk G-\rk K=n/2-1$. Hence Theorem~\ref{theo01} applies as above.

\item The spaces $\SU(2n)/\Sp(n)$. The argument is basically the same as for $n$ odd in the previous case: The algebra $H^*(\B\SU(2n))$ is concentrated in even degrees $4, 6, \ldots, 4n$ whereas $H^*(\B \Sp(n))$ is concentrated in degree divisible by four, namely $4,8, \ldots, 4n$. Hence, again merely for degree reasons $n-1$ of the generators of $H^*(\B \SU(n))$, namely the ones in degrees $6,10,\ldots, 4n-2$ map to zero under the morphism induced by the group inclusion. Since the corank of $\Sp(n)$ in $\SU(2n)$ is exactly $n-1$, we deduce that the isotropy action is equivariantly formal by Theorem~\ref{theo01}.

\item The spaces $\SO(p+q)/\SO(p)\times \SO(q)$ with both $p,q$ odd. The stabiliser inclusion is blockwise. The subgroup has corank $1$. Hence it suffices to observe that the $(p+q)/2$-th elementary symmetric polynomial in the $t_i^2$, the formal roots of $H^*(\B \SO(p+q))$, the Pontryagin class corresponding to $t_1\cdots t_{(p+q)/2}$ maps to zero under the map induced by the inclusion on the cohomology of classifying spaces. This follows from the fact that this map is induced by the projection $t_{(p+1)/2}\mapsto 0$. By Theorem~\ref{theo01} the isotropy action is equivariantly formal.
\item The space $\E_6/\pp\Sp(4)$. Since the cohomology of $\E_6$  is generated in degrees $3, 9, 11, 15, 17, 23$ and the one of the denominator group in degrees $3,7,11,15$ we derive that for degree reasons the generators of $\E_6$ in degrees $9$ and $17$ map to zero under the differential of the standard model. Theorem~\ref{theo01} yields equivariant formality, since also the corank of the two groups is $2$.
\item The space $\E_6/\F_4$. Since the cohomology of $\E_6$  is generated in degrees $3, 9, 11, 15, 17, 23$ an the one of $\F_4$ in degrees $3,11,15,23$ we see as above that the cohomology algebra of $\E_6/\F_4$ is a free algebra generated in degrees $9$ and $17$. Indeed, for degree reasons the differentials of the generators in these degrees vanish. Moreover, the space can only have finite dimensional cohomology if the other generators form a contractible algebra---we use reasoning analogous to that employed in Section~\ref{subsec05}. Due to Corollary~\ref{cor02} the isotropy action is equivariantly formal.
\end{itemize}




\def\cprime{$'$}

\pagebreak \

\vfill

\begin{center}
\noindent
\begin{minipage}{\linewidth}
\small \noindent \textsc
{Manuel Amann} \\
\textsc{Institut f\"ur Mathematik}\\
\textsc{Differentialgeometrie}\\
\textsc{Universit\"at Augsburg}\\
\textsc{Universit\"atsstra\ss{}e 14 }\\
\textsc{86159 Augsburg}\\
\textsc{Germany}\\
[1ex]
\textsf{manuel.amann@math.uni-augsburg.de}\\
\textsf{www.uni-augsburg.de/de/fakultaet/mntf/math/prof/dif\/f/team/dr-habil-manuel-amann/}
\end{minipage}

\hfill\\
\hfill\\

\noindent
\begin{minipage}{\linewidth}
\small \noindent \textsc
{Andreas Kollross} \\
\textsc{Universit\"at Stuttgart}\\
\textsc{Institut f\"ur Geometrie und Topologie}\\
\textsc{Lehrstuhl f\"ur Geometrie}\\
\textsc{Pfaffenwaldring 57 }\\
\textsc{70569 Stuttgart}\\
\textsc{Germany}\\
[1ex]
\textsf{andreas.kollross@mathematik.uni-stuttgart.de}\\
\textsf{https://www.igt.uni-stuttgart.de/team/Kollross/}
\end{minipage}
\end{center}

\end{document}